\documentclass[a4paper, 10pt]{article}
\usepackage[nottoc]{tocbibind}
\usepackage{amssymb, latexsym}
\usepackage{amsmath}
\usepackage{amsthm}
\usepackage{color}
\usepackage[utf8]{inputenc}
\usepackage{hyperref}
\usepackage{enumitem}
\usepackage{layout}
\usepackage[a4paper, textwidth=375pt]{geometry}

\allowdisplaybreaks

\newcommand{\dd}{\mathrm{d}} %d for differentials
\newcommand{\R}{\mathbb{R}} % real numbers
\newcommand{\ph}{\varphi} %prefered phi
\newcommand{\norm}[1]{\Vert#1\Vert}  % norm
\newcommand{\dist}{\mathrm{dist}} % distance function
\newcommand{\Sp}{\mathcal{S}} % script S for spheres
\newcommand{\abs}[1]{\left|#1\right|} % absolute value
\newcommand{\leqs}{\leqslant} % less than or equal
\newcommand{\geqs}{\geqslant} % greater  than or equal
\newcommand{\ep}{\varepsilon} %shorthand epsilon
\newcommand{\subs}{\subseteq} % shorthand subset
\newcommand{\p}{\partial} % shorthand partial
\newcommand{\hf}{\frac{1}{2}}
\newcommand{\D}{\mathbb{D}}

\theoremstyle{plain}
\newtheorem{thm}{Theorem}[section]
\newtheorem{lem}[thm]{Lemma}

\newtheorem{prop}[thm]{Proposition}

\theoremstyle{definition}

\theoremstyle{remark}

\numberwithin{equation}{section}

\begin{document}
	\title{Uniqueness of Weak Solutions of the Plateau Flow}
	\author{Christopher Wright\footnote{Mathematical Institute, University of Oxford. Email: christopher.wright@maths.ox.ac.uk}}
	\maketitle

\begin{abstract}
	In this paper, we study the uniqueness of weak solutions of the Plateau flow, which was first introduced by Wettstein as a half-Laplacian heat flow and recently studied by Struwe using alternative methods. This geometric gradient flow is of interest due to its links with free boundary minimal surfaces and the Plateau problem. We obtain uniqueness of weak solutions of this flow under a natural condition on the energy, which answers positively a question raised by Struwe.
\end{abstract}
\section{Introduction}
In this paper, we consider the uniqueness of solutions to a gradient flow of maps from the unit circle $\Sp^1$ into a closed manifold of any dimension, $(N,h)$, embedded smoothly into $\R^n$. This flow was introduced and studied by Wettstein in \cite{Wettstein_1,Wettstein_2,Wettstein_3} as the equation
\begin{equation}\label{eq_wettstein_flow}
	\p_t u + P_u\left((-\Delta)^\hf u\right) = 0 \text{ on }\Sp^1\times[0,T)
\end{equation}
where $P_x$ is the orthogonal projection of $\R^n$ onto the tangent space of $N$ at $x$. This is the $L^2$-gradient flow of the \emph{half-energy}
\begin{equation*}
	E_{\hf}(u) := \hf\int_{\Sp^1}\abs{(-\Delta)^{\frac{1}{4}}u}^2\dd s
\end{equation*}
within the function class $u \in H^\hf(\Sp^1;N)$. Note that this is an extrinsic energy, depending on the choice of embedding $N \hookrightarrow \R^n$. Critical points of this energy are called \emph{half-harmonic maps}, and are characterised as solutions to
\begin{equation*}
	P_u\left((-\Delta)^\hf u\right) = 0 \text{ on } \Sp^1.
\end{equation*}
The half-energy and half-harmonic maps were first introduced by Da Lio and Rivière in \cite{Da_Lio_Riviere_half_harmonic_intro}, and a special case of the gradient flow is studied by Sire, Wei and Zheng in \cite{Sire_Wei_Zheng_flow_paper}. The geometric motivation behind studying half-harmonic maps comes from their connection with free boundary minimal surfaces. It is observed in \cite{Milot_Sire_frac_lap_representation} and in \cite{Da_Lio_Martinazzi_Riviere} that if $u:\p\D \rightarrow N \subs \R^n$ is a non-constant half-harmonic map, where $\D \subs \R^2$ is the open unit disc, then the harmonic extension $u:\overline{\D} \rightarrow \R^n$ parametrises a free boundary branched minimal immersion.

In \cite{Wettstein_2,Wettstein_3}, Wettstein established the existence of a weak solution $u \in H_{\text{loc}}^1([0,\infty);L^2(\Sp^1;N))\cap L^\infty([0,\infty);H^\hf(\Sp^1;N))$ to \eqref{eq_wettstein_flow} for two cases. First, in the case of an arbitrary target $(N,h)$, existence of weak solutions is established for initial data with small energy. Second, if the target manifold is a sphere, then existence of a weak solution is established for arbitrary initial data. Moreover, these solutions have non-increasing energy and are smooth away from finitely many singular times. Wettstein also obtained a uniqueness result in \cite{Wettstein_2}, but this applies only to solutions with small energy throughout, leaving open the question of uniqueness of general weak solutions.

In \cite{Struwe_Plateau_Paper}, Struwe studies this geometric flow using very different techniques. Whereas Wettstein uses the methods of fractional calculus, Struwe uses the alternative characterisation of the half-Laplacian as
\begin{equation}\label{eq_half_laplace_D_to_N}
	(-\Delta)^\hf u = \p_\nu u
\end{equation}
where $\p_\nu$ is the Dirichlet-to-Neumann operator associated to the Laplacian on $\D$. This has previously been used for example by Millot and Sire in \cite{Milot_Sire_frac_lap_representation} and Moser in \cite{Moser_semiharmonic_maps} and is a special case of a more general theory of representing fractional operators, developed by Caffarelli and Silvestre in \cite{Caffarelli_Silvestre_Fraction_to_Local_Representation}. The relation \eqref{eq_half_laplace_D_to_N} transforms the gradient flow to the following, 
\begin{equation}\label{eq_plateau_flow_strong}
	\p_tu + P_u(\p_\nu u) = 0 \text{ on } \p\D\times[0,T)
\end{equation}
which is what Struwe defines as the Plateau flow. One advantage of this approach is that it allows for the localisation of many of the arguments by studying the harmonic extension of a map defined on $\p\D$ to all of $\overline{\D}$. For this reason, we do not distinguish in notation between a function defined on $\p\D$ and its harmonic extension defined on $\overline{\D}$. This also allows the half-energy to be replaced by the better understood Dirichlet energy
\begin{equation*}
	E(u) := \hf\int_{\D}\abs{\nabla u}^2\dd x
\end{equation*}
where equation (1.5) in \cite{Struwe_Plateau_Paper} shows that this is equal to the half energy in this setting. We refer to \cite{Struwe_Plateau_Paper} for further details behind the flow and the motivation for using Struwe's approach. One detail we note is that Struwe assumes that $(N,h)$ is embedded into $\R^n$ in such a way that the normal bundle is parallelisable. Struwe makes this assumption so as to make use of a globally defined function $\dist_N : \R^n \rightarrow \R^m$ (defined in Section 1.9 of \cite{Struwe_Plateau_Paper}). This assumption is satisfied in particular for hypersurfaces and curves in $\R^n$. Since we build upon Struwe's results, we also make this restriction.

Struwe obtained the existence of weak solutions to \eqref{eq_plateau_flow_strong} with non-increasing energy for arbitrary initial data $u_0 \in H^\hf(\p\D;N)$. Moreover, Struwe gave a more detailed description of the regularity at singular times, ensuring that his solution is smooth for $t>0$ except at finitely many spacetime points. Further, Struwe obtained a uniqueness result for arbitrary initial data, however with extra regularity of the flow assumed. In his paper, Struwe states the result for solutions which are smooth for positive time. However, the proof works for the following statement if we interpret quantities in a suitable weak sense. 

\begin{thm}[{\cite[Theorem 7.1]{Struwe_Plateau_Paper}}]\label{thm_Struwe_uniqueness_weakened}
	Let  $u,v$ be weak solutions to the Plateau Flow \eqref{eq_plateau_flow_strong} on the time interval $[0,T)$, $T \leqs \infty$, with the same initial data $u_0 \in H^\hf(\p\D;N)$. Suppose that additionally $u,v \in H^1(\p\D\times[0,T);\R^n)$ (i.e. we also have $\p_s u,\p_s v \in L^2(\p\D\times[0,T);\R^n)$). Then $u = v$.
\end{thm}

Since the flow can be uniquely extended past each singular time, this provides uniqueness for global solutions which have regularity $H^1(\p\D\times(T_1,T_2);\R^n)$ between the singular times. Following this, and by comparison with the theory of harmonic map flow, it is natural to consider if we can obtain uniqueness of solutions within a more general class of functions. Indeed, Struwe directly raises the question of whether uniqueness holds in the larger space of weak solutions with non-increasing energy. We can answer this question positively in the following theorem, which in fact also allows for small increases in energy. Here and after, we call the solution obtained by Wettstein/Struwe the \emph{almost smooth solution} associated to a particular initial map.

\begin{thm}\label{thm_uniqueness_new_1}
	Let $(N,h)$ be a smooth manifold which is smoothly embedded into $\R^n$ with parallelisable normal bundle. There exists $\ep_0 > 0$, depending only on $(N,h)$, such that if $u$ is a weak solution to the Plateau flow \eqref{eq_plateau_flow_strong} defined on the time interval $[0,T)$, $T \leqs \infty$, satisfying
	\begin{equation}\label{ineq_small_energy_jump_condition}
		\limsup\limits_{s\searrow t} E(u(s)) < E(u(t)) + \ep_0
	\end{equation}
	for all $t\in [0,T)$, then $u$ is equal to the almost smooth solution with the same initial data. 
\end{thm}

This result provides a parallel with the theory of harmonic map flow. In that setting, Struwe constructed a global weak solution in \cite{Struwe_HMF_paper} with a uniqueness result similar to Theorem \ref{thm_Struwe_uniqueness_weakened}, in that uniqueness is obtained for solutions under additional regularity assumptions. However, it was subsequently shown by Freire in \cite{Freire_HMF_Uniqueness} that uniqueness holds in a weaker class of solutions subject only to the condition that the energy is non-increasing. It is perhaps surprising that energy monotonicity needs to be assumed at all for a gradient flow, but the existence of backwards bubbles due to Topping, \cite{Topping_Reverse_Bubbling}, shows that energy monotonicity can fail if we look in the class of weak solutions with no additional assumptions. Following this, Rupflin strengthened Freire's result in \cite{Rupflin_Harmonic_Map_Flow_Uniqueness} to show that uniqueness holds even when we allow for small increases in energy, with the smallness needed to exclude backwards bubbles. We adapt the methods of Rupflin for our proofs below, and note that our results provide analogues of Theorems 1.1 and 1.2 from \cite{Rupflin_Harmonic_Map_Flow_Uniqueness} in the Plateau flow setting.

Our Theorem \ref{thm_uniqueness_new_1} is enough to cover the case of weak solutions with non-increasing energy, but since we do not have any lower bound on the value of $\ep_0$, we cannot use this result to conclude that non-uniqueness must be caused by backwards bubbling. To get closer to this, we prove the following slightly strengthened theorem. Note that $\ep^* \geqs \ep_0$, so the assumptions are indeed weaker than in Theorem \ref{thm_uniqueness_new_1}.

\begin{thm}\label{thm_uniqueness_new_2}
	Let $(N,h)$ be a smooth manifold which is smoothly embedded into $\R^n$ with parallelisable normal bundle. Let $\ep^* > 0$ be defined as
	\begin{equation*}
		\ep^* := \inf\lbrace E(u) : u:\p\D \rightarrow N \text{ non-constant and half-harmonic}\rbrace
	\end{equation*}
	Suppose $u$ is a weak solution to the Plateau flow \eqref{eq_plateau_flow_strong} on the time interval $[0,T)$, $T\leqs\infty$, satisfying the two conditions
	\begin{equation}\label{ineq_bubble_energy_jump_condition}
		\limsup\limits_{s\searrow t} E(u(s)) < E(u(t)) + \ep^*
	\end{equation}
	for all $t \in [0,T)$ and that the set
	\begin{equation}\label{est_finite_bad_jumps}
		S := \{t \in [0,T): \limsup\limits_{s\searrow t} E(u(s)) \geqs E(u(t)) + \ep_0\}
	\end{equation}
	has no accumulation points, where $\ep_0$ is the constant from Theorem \ref{thm_uniqueness_new_1}. Then $u$ is equal to the almost smooth solution with the same initial data.
\end{thm}
Note that the second condition \eqref{est_finite_bad_jumps} is satisfied in particular if $t \mapsto E(t)$ has locally finite total variation.

This paper is structured as follows. In Section \ref{sec_thm_1}, we provide the proof of Theorem \ref{thm_uniqueness_new_1}. For this, we recall some regularity results of Wettstein and Struwe. We then combine these with estimates on the concentration of energy to show that under the conditions of Theorem \ref{thm_uniqueness_new_1}, the solution has the regularity required by Theorem \ref{thm_Struwe_uniqueness_weakened}. Then in Section \ref{sec_thm_2}, we prove Theorem \ref{thm_uniqueness_new_2} using the analysis of bubble formation from \cite{Struwe_Plateau_Paper} to show that concentration of energy backwards in time causes a half-harmonic bubble to form.

\paragraph{Acknowledgements}
I would like to thank my supervisor Melanie Rupflin for her support and advice and also Michael Struwe for insightful discussions on this problem. This research was supported by the Engineering and Physical Sciences Research Council. 

\section{Proof of Theorem \ref{thm_uniqueness_new_1}}\label{sec_thm_1}
First, let us set out some conventions and notations for the remainder of the paper. By a \emph{weak solution} of the Plateau flow, we mean a function in $H_{\text{loc}}^1([0,T);L^2(\Sp^1;N))\cap L^\infty([0,T);H^\hf(\Sp^1;N))$ which weakly solves the equation \eqref{eq_plateau_flow_strong}, where $T \in (0,\infty]$. We use the notation 
\begin{equation*}
	E(u;\Omega) := \hf\int_{\Omega\cap \D}\abs{\nabla u}^2\dd x
\end{equation*}
for the energy of $u \in H^1(\D)$ on a subset $\Omega \subs \R^2$. Also we write $B_r(x)$ for an open ball centred at $x$ with radius $r$.

Next we state some existing regularity results from the literature which we make use of below. First of all, we have the following qualitative $H^1$ regularity result due to Wettstein. Note that this result is originally stated with the half-Laplacian, but by \eqref{eq_half_laplace_D_to_N}, this is equivalent to the Dirichlet-to-Neumann operator that we use. 
\begin{prop}[{\cite[Lemma 3.8]{Wettstein_1}},{\cite[Proposition 4.1]{Wettstein_2}}]\label{prop_Wettstein_static_regularity}
	Suppose that $u \in H^\hf(\p\D;N)$ and $f \in L^2(\p\D;\R^n)$ satisfy
	\begin{equation}\label{eq_inhomogeneous}
		P_u(\p_{\nu} u) = f.
	\end{equation}
	Then $u \in H^1(\p\D;N)$.
\end{prop}
We note that this result is found as \cite[Lemma 3.8]{Wettstein_1} in the case where $N$ is a sphere, and the proof in the general case can be extracted from the proof of \cite[Proposition 4.1]{Wettstein_2} by replacing $\p_t u$ with $f$.

Next we need a quantitative $H^1$ estimate for the same equation. The following, due to Struwe, is originally stated and proved for smooth functions $u$, but as remarked in that paper, the proof works just the same for the statement given below.

\begin{prop}[{\cite[Proposition 3.4]{Struwe_Plateau_Paper}}]\label{prop_quantative_static_H1_regularity_global}
	There exists a constant $\ep_0 > 0$, depending only on $(N,h)$, such that if $u \in H^1(\p\D;N)$ solves \eqref{eq_inhomogeneous} with $f \in L^2(\p\D;\R^n)$ and $0 < r < \hf$ satisfies
	\begin{equation*}
		\sup\limits_{x \in \D}E(u;B_r(x))\leqs \ep_0
	\end{equation*}
	then
	\begin{equation*}
		\norm{\p_s u}_{L^2(\p\D;\R^n)}^2 \leqs 	C\left(\norm{f}_{L^2(\p\D;\R^n)}^2 + E(u)\right)
	\end{equation*}
	where $C$ depends only on $r$.
\end{prop}

Now we can give the proof of our first uniqueness result.

\begin{proof}[Proof of Theorem \ref{thm_uniqueness_new_1}]
	Let $u$ be a weak solution of the Plateau flow satisfying \eqref{ineq_small_energy_jump_condition} with $\ep_0$ taken to be the same as the $\ep_0$ from Proposition \ref{prop_quantative_static_H1_regularity_global}. First of all, we note that for almost all times, $u(t)$ is a weak solution to the stationary equation \eqref{eq_inhomogeneous} with $f = -\p_t u(t) \in L^2(\p\D;\R^n)$. Hence, we can apply Proposition \ref{prop_Wettstein_static_regularity} to obtain that for almost all times, $u(t) \in H^1(\p\D;N)$. 
	
	Using this, we would like to apply the quantitative $H^1$ estimate, for which we use the following lemma (based upon the ideas of Lemma 3.3 from \cite{Rupflin_Harmonic_Map_Flow_Uniqueness}).
	\begin{lem}\label{lem_no_energy_Concentration}
		Let $u$ be a weak solution to the Plateau flow \eqref{eq_plateau_flow_strong} on $[0,T)$ and suppose that at some time $t_0 \in [0,T)$, $u$ satisfies
		\begin{equation*}
			\limsup\limits_{s\searrow t_0} E(u(s)) < E(u(t_0)) + \ep
		\end{equation*}
		for some $\ep>0$. Then there exists $t_1 > t_0$ and $r > 0$ such that
		\begin{equation*}
			\sup\limits_{x \in \overline{\D},t_0\leqs t \leqs t_1}E(u(t);B_r(x)) \leqs \ep.
		\end{equation*}
	\end{lem}
	\begin{proof}
		Let $\rho > 0$ satisfy
		\begin{equation*}
			\limsup\limits_{s\searrow t_0} E(u(s)) \leqs E(u(t_0)) + \ep - \rho.
		\end{equation*}
		Then we select a covering $\{B_{r}(x_i): 1 \leqs i \leqs m\}$ of $\overline{\D}$ satisfying
		\begin{equation*}
			E(u(t_0);B_{2r}(x_i)) \leqs \hf\rho
		\end{equation*}
		for each $i$. We now claim that there exists $t_1 > t_0$ such that for all $t \in [t_0,t_1]$ and for all indices $i$,
		\begin{equation*}
			E(u(t);B_{2r}(x_i)) \leqs \ep
		\end{equation*}
		Arguing by contradiction, suppose that there is a sequence of times $t_k \searrow t_0$ and a sequence of indices $i_k$ such that
		\begin{equation*}
			E(u(t_k);B_{2r}(x_{i_k})) > \ep
		\end{equation*}
		By passing to a subsequence, we assume without loss of generality that $i_k = 1$ and then estimate
		\begin{align}
			\ep &\leqs \limsup\limits_{k \rightarrow \infty}E(u(t_k);B_{2r}(x_1))\nonumber\\
			&= \limsup\limits_{k \rightarrow \infty}\left(E(u(t_k)) - E(u(t_k);\D\setminus B_{2r}(x_1)) \right)\nonumber\\
			&= \limsup\limits_{k \rightarrow \infty}E(u(t_k)) - \liminf\limits_{k \rightarrow \infty}E(u(t_k);\D\setminus B_{2r}(x_1))\nonumber\\
			&\leqs E(u(t_0)) + \ep - \rho - \liminf\limits_{k \rightarrow \infty}E(u(t_k);\D\setminus B_{2r}(x_1))\label{eq_liminf_line}
		\end{align}
		with the final inequality following from the choice of $\rho$. Next we estimate the $\liminf$ term. For this, we claim that along any sequence $\tilde{t}_k \rightarrow t_0$ and for any subset $\Omega\subs\D$,
		\begin{equation}\label{ineq_liminf_est}
			\liminf\limits_{k \rightarrow \infty}E(u(\tilde{t}_k);\Omega) \geqs E(u(t_0);\Omega).
		\end{equation}
		To prove this, we note that as $u$ is a weak solution, $u(\tilde{t}_k)$ is a bounded sequence in $H^{\frac{1}{2}}(\p\D;N)$ and hence the harmonic extensions form a bounded sequence in $H^1(\D;\R^n)$, and consequently also in $H^1(\Omega;\R^n)$. Therefore, on passing to a subsequence we can assume that
		\begin{align*}
			u(\tilde{t}_k) &\rightarrow u_\infty \text{ strongly in } L^2(\Omega;\R^n)\\
			\nabla u(\tilde{t}_k) &\rightharpoonup \nabla u_\infty \text{ weakly in } L^2(\Omega;\R^n)
		\end{align*}
		We have that $u(\tilde{t}_k) \rightarrow u(t_0)$ in $L^2(\Omega;\R^n)$ since $t \mapsto u(t)$ is continuous as a map from $[0,T]$ to $L^2(\D;\R^n)$. Hence $u_\infty = u(t_0)$, and so \eqref{ineq_liminf_est} follows by weak convergence of the gradients.
		
		Returning to \eqref{eq_liminf_line}, by choosing $\tilde{t}_k = t_k$ and $\Omega = \D\setminus B_{2r}(x_1)$, we get
		\begin{align*}
			\ep &\leqs E(u(t_0)) + \ep - \rho - \liminf\limits_{k \rightarrow \infty}E(u(t_k);\D\setminus B_{2r}(x_1))\\
			& \leqs E(u(t_0)) + \ep - \rho - E(u(t_0);\D\setminus B_{2r}(x_1))\\
			&= E(u(t_0);B_{2r}(x_1)) + \ep - \rho\\
			&\leqs \ep - \frac{1}{2}\rho
		\end{align*}
		which gives the desired contradiction. Since any ball of radius $r$ lies in $B_{2r}(x_i)$ for some $i$, we are done.
	\end{proof}
	We now use this to finish the proof of Theorem \ref{thm_uniqueness_new_1} by noting that for each $t_0 \in [0,T)$, on choosing $\ep=\ep_0$, Lemma \ref{lem_no_energy_Concentration} gives us a $t_1 > t_0$ and an $r>0$ such that
	\begin{equation*}
		\sup\limits_{x \in \D,t_0\leqs t \leqs t_1}E(u(t); B_r(x)) \leqs \ep_0.
	\end{equation*}
	Hence we can apply Proposition \ref{prop_quantative_static_H1_regularity_global} for almost every $t \in [t_0,t_1]$ and with this $r$ to get
	\begin{equation*}
		\norm{\p_s u(t)}_{L^2(\p\D;\R^n)}^2 \leqs C\left(\norm{\p_t u(t)}_{L^2(\p\D;\R^n)}^2 + E(u(t))\right)
	\end{equation*}
	for a constant $C$ depending only on $r$ and $\ep_0$. We can integrate over $[t_0,t_1]$ to get $u \in H^1(\p\D\times[t_0,t_1];\R^n)$. Hence we can apply Theorem \ref{thm_Struwe_uniqueness_weakened} to get that $u$ is the almost smooth solution with initial data $u(t_0)$ on $[t_0,t_1]$. Since this applies to all $t_0 \in [0,T)$ and we can always continue the flow uniquely by extracting the unique weak limit in $H^\hf(\p\D;N)$, $u$ must be the almost smooth solution on all of $[0,T)$.
\end{proof}

\section{Proof of Theorem \ref{thm_uniqueness_new_2}}\label{sec_thm_2}
Key to the proof of Theorem \ref{thm_uniqueness_new_2} is the bubbling analysis contained in Section 8 of \cite{Struwe_Plateau_Paper}, where Struwe obtained results on the behaviour of solutions of the Plateau flow approaching a singularity forwards in time. In essence, this involves rescaling by a factor $r_k$ in both space and time around the singularity and then selecting a time sequence such that these rescaled maps converge suitably. What is clear from Struwe's proofs however is that this only depends on the Plateau flow equation via estimates of the form \eqref{ass_static_bubble_energy_conc} and \eqref{ass_static_bubble_normal_control}. Therefore it is possible, and useful for us, to distil from Struwe's work the following more flexible proposition, which gives sufficient conditions for the formation of a half-harmonic bubble along a sequence of maps. 

\begin{prop}[{\cite[Section 8]{Struwe_Plateau_Paper}}]\label{prop_bubble_formation}
	Suppose that $u_k: \p\D \rightarrow N$ is a sequence of smooth maps which are extended harmonically to the interior of $\D$ and which have uniformly bounded energy. Suppose that there are sequences $x_k \in \overline{\D}$ and $r_k \searrow 0$ and a constant $\delta>0$ such that
	\begin{align}
		E(u_k;B_{r_k}(x_k)) \geqs \delta &\text{ for all } k\label{ass_static_bubble_energy_conc}\\
		r_k^{\frac{1}{2}}\norm{P_{u_k}(\p_{\nu} u_k)}_{L^2(\p \D)} \rightarrow 0 &\text{ as } k \rightarrow 0.\label{ass_static_bubble_normal_control}
	\end{align}
	Then there exists a subsequence, a non-constant half-harmonic map $\bar{u} \in H^{\hf}(\p\D;N)$ and a sequence of smooth conformal bijections $\Phi_k:\overline{\D} \rightarrow \overline{\D}$ converging weakly in $H^1(\D)$ to a constant map taking value $x_0 \in \p\D$ such that
	\begin{align}
		u_k \circ \Phi_k &\rightharpoonup \bar{u} \text{ weakly in } H^1(\overline{\D})\label{eq_bubble_weak_con}\\
		u_k \circ \Phi_k &\rightarrow \bar{u} \text{ strongly in } H_{\text{loc}}^1(\overline{\D}\setminus\{p\})\label{eq_bubble_strong_con}
	\end{align}
	where $p \in \p\D$.
\end{prop}
Before giving the proof of Theorem \ref{thm_uniqueness_new_2}, we briefly unpack the geometry behind Proposition \ref{prop_bubble_formation}. The condition \eqref{ass_static_bubble_energy_conc} is saying that energy is concentrating in smaller and smaller regions of $\overline{\D}$ and the condition \eqref{ass_static_bubble_normal_control} is the analogue of the tension bounds found in the bubbling theory of almost harmonic maps. The mappings $\Phi_k$ are mapping more and more of the disc $\D$ into a small region around $x_0$, which is where the bubble is forming. The region away from the bubble formation is collapsed by $\Phi_k^{-1}$ to the point $p$ in the limit. For our purposes, the weak convergence \eqref{eq_bubble_weak_con} is sufficient, but it is straightforward to get the stronger convergence \eqref{eq_bubble_strong_con} by ensuring that the $\Phi_k$ are extracting the top level bubble which is forming, i.e. the bubble forming at the smallest scale.
 
\begin{proof}[Proof of Theorem \ref{thm_uniqueness_new_2}]
	Let $u$ be a weak solution of the Plateau flow satisfying \eqref{ineq_bubble_energy_jump_condition} and \eqref{est_finite_bad_jumps}. The strategy of the proof is then to show that for each $s \in S$, there is some interval $[s,\tau]$ on which $u$ equals the almost smooth solution with initial data $u(s)$. Since on $[0,T)\setminus S$, the condition \eqref{ineq_small_energy_jump_condition} for Theorem \ref{thm_uniqueness_new_1} holds, this is sufficient to conclude that $u$ is the almost smooth solution with initial data $u(0)$ on all of $[0,T)$.
	
	So, we assume without loss of generality that $0 \in S$ and aim to construct an interval $[0,\tau]$ as outlined above. As $S$ has no accumulation points, there is then some time $T_1 > 0$ such that $(0,T_1] \subs [0,T)\setminus S$. Then we know from Theorem \ref{thm_uniqueness_new_1} that for any $0 < t_0 < T_1$, $u$ must equal the almost smooth solution with initial data $u(t_0)$ on $[t_0,T_1]$. By reducing $T_1$ to before the first positive singular time, we then additionally assume that $u$ is smooth on $[t_0,T_1]$. 
	
	Now, we claim that there exists some $r>0$ and $t_1 \in (0,T_1]$ such that for all $t \in [0,t_1]$ and for all $x \in \overline{\D}$
	\begin{equation*}
		E(u(t);B_r(x)) \leqs \ep_0.
	\end{equation*}
	Once this claim is proved, we can repeat the argument from the proof of Theorem \ref{thm_uniqueness_new_1} to get that $u \in H^1(\p\D\times[0,t_1];\R^n)$ and so apply Struwe's uniqueness theorem, Theorem \ref{thm_Struwe_uniqueness_weakened}.
	
	To prove this claim, we argue by contradiction. So suppose that there exist sequences $t_k \searrow 0$, $r_k \rightarrow 0$ such that
	\begin{equation*}
		\sup\limits_{x \in \overline{\D}}E(u(t_k);B_{r_k}(x)) > \ep_0.
	\end{equation*}
	We can then find a sequence $x_k$ such that for each $k$,
	\begin{equation*}
		E(u(t_k);B_{r_k}(x_k)) > \ep_0.
	\end{equation*}
	By passing to a subsequence, we assume that $x_k \rightarrow x_0$. We then fix $\rho>0$ which satisfies
	\begin{equation*}
		\limsup\limits_{s\searrow 0} E(u(s)) \leqs E(u(0)) + \ep^* - \rho
	\end{equation*}
	and select $r_0>0$ such that
	\begin{equation*}
		E(u(0);B_{r_0}(x_0)) \leqs \frac{1}{2}\rho
	\end{equation*}

	Next, we construct a new sequence of times, $\tilde{t}_k \searrow 0$, along which we can apply Proposition \ref{prop_bubble_formation} to extract a half-harmonic bubble. First, in order to satisfy the bound \eqref{ass_static_bubble_normal_control} on $P_{u_k}(\p_\nu u_k)$, we use that $\p_t u \in L^2(\p\D \times [0,T))$ and hence
	\begin{equation*}
		\int_{t_k}^{t_k+r_k}\int_{\p\D}\abs{P_u(\p_{\nu} u)}^2\dd s\dd t \rightarrow 0
	\end{equation*}
	as $k \rightarrow \infty$. Therefore we can choose some $\tilde{t}_k \in [t_k,t_k + r_k]$ satisfying
	\begin{equation*}
		\int_{\p\D}\abs{P_u(\p_{\nu} u(\tilde{t}_k))}^2\dd s\leqs \frac{1}{r_k}\int_{t_k}^{t_k+r_k}\int_{\p\D}\abs{P_u(\p_{\nu} u)}^2\dd s\dd t
	\end{equation*}
	to obtain \eqref{ass_static_bubble_normal_control}. To see that the energy concentration condition \eqref{ass_static_bubble_energy_conc} still holds for this new sequence of times, we can use the following lemma, which is extracted from analysis done in Section 8 of \cite{Struwe_Plateau_Paper}.
	\begin{lem}[{\cite[Section 8]{Struwe_Plateau_Paper}}]\label{lem_local_energy_bound}
		Let $u$ be a weak solution to the Plateau flow which is smooth on a time interval $[t_0,t_1]$. Let $\ph$ be a smooth cut-off function which is supported on $B_{2r}(x_0)$ and is identically $1$ on $B_r(x_0)$. Then
		\begin{equation*}
			\abs{\int_\D \abs{\nabla u (t_1)}^2\ph^2 \dd x - \int_\D \abs{\nabla u (t_0)}^2\ph^2 \dd x} \leqs C\int_{t_0}^{t_1}\int_{\p\D}\abs{\p_t u}^2 \dd s\dd t
		\end{equation*} 
		where $C$ depends only on the upper bound on the energy and $r\sup\limits_{x \in \overline{\D}}\abs{\nabla \ph(x)}$.
	\end{lem}
	Choosing $\ph_k$ to be centred at $x_0$ with $r = r_k$ and the time interval $[t_k,\tilde{t}_k]$, from Lemma \ref{lem_local_energy_bound} we get
	\begin{equation*}
		\int_{\D}\abs{\nabla u(\tilde{t}_k)}^2\ph_k^2 \dd x \geqs \int_{\D}\abs{\nabla u(t_k)}^2\ph_k^2 \dd x - o(1)
	\end{equation*}
	as $k \rightarrow \infty$. Hence using $\delta = \hf\ep_0$, we have that the sequence of maps $u(\tilde{t}_k)$ satisfy the conditions of Proposition \ref{prop_bubble_formation} and so let $\bar{u} \in H^1(\D;\R^n), \Phi_k \in C^\infty(\overline{\D};\overline{\D})$ be the resulting functions. By conformal invariance of the Dirichlet energy, we have
	\begin{equation*}
		E(u(\tilde{t}_k);B_{r_0}(x_0)) = E(u(\tilde{t}_k) \circ \Phi_k;\Phi_k^{-1}(B_{r_0}(x_0))).
	\end{equation*}
	Since $\Phi_k \rightharpoonup x_0$ in $H^1(\D)$, we have that $m(\D\setminus \Phi_k^{-1}(B_{r_0}(x_0))) \rightarrow 0$ as $k \rightarrow \infty$, $m$ being the Lebesgue measure. So for any fixed $\eta>0$, we can pass to a subsequence such that $\Omega = \bigcap\limits_{k\geqs 1}\Phi_k^{-1}(B_{r_0}(x_0))$ has $m(\Omega) \geqs \pi - \eta$. From this, we can use the weak convergence of $u(\tilde{t}_k)\circ\Phi_k \rightharpoonup \bar{u}$ in $H^1(\D)$ to estimate
	\begin{equation*}
		\liminf\limits_{k\rightarrow\infty}E(u(\tilde{t}_k);B_{r_0}(x_0)) \geqs \liminf\limits_{k\rightarrow\infty}E(u(\tilde{t}_k) \circ \Phi_k;\Omega)
		\geqs E(\bar{u};\Omega).
	\end{equation*}
	Since we can choose any $\eta>0$, we obtain
	\begin{equation*}
		\liminf\limits_{k\rightarrow\infty}E(u(\bar{t}_k);B_{r_0}(x_0)) \geqs E(\bar{u}) \geqs \ep^*
	\end{equation*}
	Combining this with the estimate \eqref{ineq_liminf_est} to relate $E(u(0))$ with $E(u(\tilde{t}_k))$, we obtain
	\begin{align*}
		E(u(0)) &\leqs E(u(0);\overline{\D}\setminus B_{r_0}(x_0)) + \frac{1}{2}\rho\\
		&\leqs \liminf\limits_{k \rightarrow \infty}E(u(\tilde{t}_k);\overline{\D}\setminus B_{r_0}(x_0)) + \frac{1}{2}\rho\\
		&\leqs \limsup\limits_{k \rightarrow \infty}E(u(\tilde{t}_k)) - \liminf\limits_{k \rightarrow \infty}E(u(\tilde{t}_k);B_{r_0}(x_0)) + \frac{1}{2}\rho\\
		&\leqs E(u(0)) - \hf\rho
	\end{align*}
	which gives the required contradiction. 
	
	From this, we argue exactly as in the proof of Theorem \ref{thm_uniqueness_new_1} to conclude that $u$ must in fact be equal to the almost smooth solution on $[0,t_1]$. This can be repeated for all times in the set $S$, and as discussed at the beginning of the proof, we can apply Theorem \ref{thm_uniqueness_new_1} away from $S$ to conclude that $u$ equals the almost smooth solution on all of $[0,T)$.
\end{proof}

\bibliographystyle{alpha}
\bibliography{Uniqueness_paper_preprint}

\end{document}